\documentclass[english, 12pt]{article}
\usepackage{amsfonts}
\usepackage{amsmath}
\usepackage{amssymb}
\usepackage{amsthm}
\usepackage{amscd}
\usepackage{amscd}
\usepackage{amsthm}
\usepackage{textcomp}
\usepackage{graphics}
\usepackage{geometry,amssymb,epsfig,graphics,float}
\newtheorem{theo}{Theorem}[section]
\newtheorem{prop}{Proposition}[section]

\numberwithin{equation}{section}
\def\md{\mathrm{~mod~}}
\def\R{\mathbb{R}}
\def\H{\mathbb{H}}
\def\A{\mathbb{A}}
\def\N{\mathbb{N}}

\def\Q{\mathbb{Q}}
\def\Z{\mathbb{Z}}
\def\D{\mathbb{D}}

\def\ND{N_{\mathbb{D}}}
\def\NA{N_{\mathbb{A}}}

\def\ZB{\mathcal{Z}_h}

\def\G{\mathcal{G}}

\newcommand{\Rj}{\mathrm{Re}}
\newcommand{\Ij}{\mathrm{Im}}
\newcommand{\eg}{\mathbf{e}}

\newcommand{\ig}{\mathbf{i}}

\newcommand{\jg}{\mathbf{j}}
\begin{document}
\bibliographystyle{alpha}

\title{ Radix form in Hyperbolic and Dual numbers   }
\author{ Sayed Kossentini\footnote {Universit\'e de Tunis El Manar, Facult\'e des Sciences de Tunis, Department of Mathematics, 2092, Tunis,  Tunisia. E-mail address:
		sakossentini@gmail.com}}

\date{}
\maketitle

\begin{abstract}
We investigate number systems for the ring of integers of hyperbolic and dual numbers. We  characterize all canonical number systems providing radix form for hyperbolic and dual numbers. Our approach allows us to get suitable bases by means of Banach lattice algebra structure. 
	
\end{abstract}

\section{Introduction}
Number system for a natural number is a positional representation in any integral basis $b\geq 2$. The most commonly used are binary and decimal systems. Using a negative basis  $b\leq -2$,  every integer $n \in \Z$ has a unique representation of the form 
\begin{equation*}
n=a_0+ \cdots+ a_l b^l, \quad a_i \in \{0, 1, \cdots, |b|-1\}.
\end{equation*}
A natural extension of these concepts to the ring  $\mathcal{O}_K$ of integers of a number field $K$ provides the following definition: let $\gamma \in \mathcal{O}_K$ and  $\mathfrak{N}$ be a complete residue system modulo $\gamma$ (i.e. for every $\alpha \in \mathcal{O}_K $ there exists a unique $r \in \mathfrak{N}$ such that $\alpha\equiv r ~( \mathrm{mod}~ \gamma)$), a pair $(\gamma, \mathfrak{N})$  is a number system  for $\mathcal{O}_K$ if every $\alpha \in \mathcal{O}_k$ has the form 
\begin{equation*}
\alpha= r_0+r_1 \gamma+ \cdots+ r_k \gamma^k,\quad r_i \in \mathfrak{N}.
\end{equation*}
This representation is unique since $\mathfrak{N}$ is a complete residue system modulo $\gamma$.  The number $\gamma$ is called base and $\mathfrak{N}$ set of digits of the number system.  $(\gamma, \mathfrak{N})$ is  called canonical number system if  $\mathfrak{N}=\mathfrak{N}_0=\{ 0,1, \cdots |N^K_\Q(\gamma)|-1\}$ , where $N^K_\Q(\alpha)$ is the  field norm of $\alpha$. The first result in this topic is due to   Knut \cite{Knut} who discovered that $-1+ \ig$ is a base for complex integers.  K\'atai and Szab\'o \cite{Kat-Sab1}  proved that the only possibles canonical number systems in Gaussian integers  are given by the basis  $q=a\pm \ig$ with $a\leq -1$. They also proved that  every complex number $z$ can be expressed in the form
\begin{equation*}
z= r_0+ \cdots+ r_l q^l+  \displaystyle \sum_{i=1}^\infty \frac{s_i}{q^i},\quad r_k, s_j \in \mathfrak{N}_0 
\end{equation*}
Their work was generalized to   a quadratic number field $K=\Q(\rho)$ by K\'atai and Kovacs \cite{Kat-Kova}  and  independently by Gilbert \cite{Gilb1}. Kovacs and Peto \cite{Kova-Pet} gave a characterization of all integral domains which have number systems. More recently, Peto and Thuswaldner \cite{Pet-Thus} investigate number systems defined over orders of number fields\\

The present  work deals with number systems for non integral domains.  To be more precise, number systems for the ring of integers of dual  and  hyperbolic numbers ( also called duplex, double, perplex and split-complex numbers). These   systems, together with complex numbers, are, up to an algebra isomorphism, the three possible two-dimensional real associative algebras, usually known as two-dimensional systems. During many years a number of researchers have been attracted by  algebraic, geometrical properties and  applications of the two-dimensional systems (see for instance \cite{Yal}, \cite{FC1}, \cite{KIS} ,\cite{SOB} and \cite{REY2}). In another direction, Wickstead \cite{WIC} studied  all possible Riesz algebra orderings on the two dimensional system using cones in $\R^2$ of the form
\begin{equation}\label{cones}
P^\beta_\alpha=\{ (x,y); x\geq 0, \alpha x\leq y\leq \beta x\}, ~ \alpha,\beta \in \R\cup\{-\infty,\infty\}, \alpha<\beta.
\end{equation}
He proved that dual and hyperbolic numbers are the only two dimensional systems that can be made into Banach lattice algebras. He also studied their representations in the Riesz algebra of regular operator. 

In this paper, we will characterize all canonical number systems providing radix form for hyperbolic and dual numbers. This can be done thanks to their Banach lattice algebra structure. 

\section{Preliminaries}
In this section, we recall  basic properties of hyperbolic and dual numbers   which we will use throughout this paper (for more details we refer the reader to \cite{GK1} and \cite{GK2}). For the used lattice concepts we recall some notions about Riesz algebras and $f$-algebras  (see for instance \cite{ZAN}, \cite{BIGA}, \cite{STEIN}).\\

A Riesz space is a real linear space $ L$ that is also a lattice ordered set (i.e. the supremum $u \vee v$ and the infimum $u\wedge v$ exist in $L$ for all $u,v \in L$ ), and the partial order satisfies: $u\leq v\Longrightarrow u+w\leq v+w$  for all $w \in L$ and $\alpha u\leq \alpha v$ for all $\alpha \in \R^+$. The set $L^+$ of all positive elements of $L$ (i.e. $u \geq 0$) is called positive cone of $L$. The absolute value of an element $a$ in a Riesz space $L$ is defined as $|u|= u\vee (-u)$.

 A real associative algebra ( with usual operations) is called Riesz algebra if its underlying linear space is a Riesz space with positive cone closed under multiplication. A Banach lattice algebra is a Riesz algebra equipped by a lattice norm (i.e. $|u| \leq |v| \Rightarrow \|u\| \leq \|v\|$) which is also sub-multiplicative.

An $f$-algebra is a Riesz algebra $A$ with the additional property: for any  $u,v, w \in A^+$, $u\wedge v=0\Rightarrow wu \wedge v= u\wedge w v=0$ 
\subsection{Hyperbolic numbers} 
In this section we summarize results of \cite{GK1} and \cite{GK2}.
The associative real algebra of hyperbolic numbers is
\begin{equation*}
	\D:=\Big\{ z=x+\jg y:\, x,y \in \R,\, \jg \notin \R;\, \jg^2=1\Big\}.
\end{equation*}
The group of units $\D^\times$ is characterized by all hyperbolic numbers $z$ with  $N_\D(z)\neq 0$, where   
\begin{equation*}
 N_\D(z)=z \bar{z}= (\Rj(z))^2 - (\Ij(z))^2.
\end{equation*}
The hyperbolic plane has an important basis $\{\eg_1,\eg_2\} $ where
\begin{equation}\label{idempbasis}
\eg_1= \frac{1+\jg}{2}, \quad \eg_2= \frac{1-\jg}{2}\Rightarrow \eg_1^2=\eg_1; \quad\eg_2^2=\eg_2; \quad\eg_1+\eg_2=1; \quad \eg_1\eg_2=0.
\end{equation}
 In this basis, each hyperbolic number $z$ can be written as 
\begin{equation}\label{decomspec}
	z=\pi_1(z) \eg_1+\pi_2(z) \eg_2,
\end{equation}
where $\pi_1(x+\jg y)=x+y$ and $ \pi_2(x+\jg y)=x-y$. From the representation \eqref{decomspec},  algebraic operations correspond to coordinate-wise operations, the square norm $N_\D(z)$  is the product $\pi_1(z) \pi_2(z)$ and the conjugation  $\bar{z}$ is given by exchanging $\pi_1(z)\leftrightarrow \pi_2(z)$. Moreover, $\D$ is an Archimedean $f$-algebra \cite{GK1} with positive cone $\D^+=\{z: \pi_1(z), \pi_2(z)\geq 0  \}$. According to \eqref{cones}, $\D^+=\{ z= x+ \jg y: (x,y) \in P^1_{-1}\}$.\\ The partial order $\leq$ on $\D$ is then 
	\begin{equation}\label{falgb}
		z, w \in \D;\, z\leq w \mbox{~if and only if~}\pi_{k}(z)\leq \pi_{k}(w) ,\quad k=1,2.
	\end{equation} 
A Banach lattice algebra norm on $\D$ is given by 	
\begin{equation*}\label{norm}
	\|z\|_{\D}:=\min \Big\{\alpha \in \R^+:\, \alpha \geq |z|\Big\}=|z|\vee \overline{|z|}.
\end{equation*}
It is the supremum norm in the idempotent basis $(\eg_1, \eg_2)$. In the standard basis $(1, \jg)$ it is given by $\| x+ \jg y\|_\D= \max\{ |x+y|, | x-y|\}$.\\
We now recall some results about the ring of hyperbolic integers $\ZB$ (see \cite{GK2}). It is  the ring of integers of $\Q(\jg)=\{ z=x+\jg y:\, x,y \in \Q\}$ given by 
	\begin{equation}\label{hintegers}
	\ZB:=\Z \eg_1+\Z \eg_2.
	\end{equation}
The units of $\ZB$ are $\{1,-1,\jg,-\jg\}$, and  each hyperbolic integer $\upsilon$ satisfies the identity
\begin{equation}\label{quadratich} 
\upsilon^2-2 \Rj(\upsilon)\upsilon+N_\D(\upsilon)=0.
\end{equation}
Divisibility in $\ZB$ is defined in a natural way. We say that $ \upsilon$ divides $\omega$, write $\upsilon| \omega$, if $ \omega=\mu \upsilon$ for some $\mu \in \ZB$. The congruence relation $\upsilon \equiv \mu \mbox{~mod~}\omega$ is defined as $ \omega| \upsilon-\mu$. From \eqref{hintegers}, divisibility in $\ZB$ is component-wise. The ordering (\ref{falgb}) on $\D$ induces the same ordering on $\ZB$.\\ One can extend usual division algorithm  to  hyperbolic integers; that is, for every pair $(\upsilon,\omega) \in \ZB\times \ZB$  with  $N_\D(\omega)\neq 0$ there exists a unique  $(\tau,\phi) \in \ZB\times \ZB $ such that
\begin{equation}\label{algodiv}
\upsilon=\tau+  \phi\omega,\quad 0\leq \tau\leq |\omega|-1.
\end{equation}
Thus, the  size of the group $\ZB/ \omega \ZB$, ($N_\D(\omega)\neq 0$) is 
\begin{equation}\label{sizegroup}
| \ZB/ \omega \ZB|=\mathrm{card}~\{ \nu \in \ZB:\, 0\leq \nu\leq |\omega|-1\}=N_\D(|\omega|).
\end{equation}
Just as complex numbers, \textit{hyperbolic Gaussian integers} are defined by 
\begin{equation*}
\G_\D= \Z[\jg]=\{ x+ \jg y: x, y \in \Z\}.
\end{equation*}  
It is a subring of $\ZB$ formed by the hyperbolic integers $\upsilon$ such that $ \upsilon\equiv 0 \md 2$ or $\upsilon\equiv 1 \md 2$. It is worth noticing here that 
the four binary classes are
$$
\ZB/2 \ZB=\{ \hat{0},\hat{1},\widehat{\eg_1},\widehat{\eg_2}\}.
$$
\subsection{Dual numbers}
The set of dual numbers $\A$ is defined in the same way as hyperbolic and complex numbers but with nilpotent imaginary unit; that is
\begin{equation*}
\A:=\Big\{ z=x+\varepsilon y:\, x,y \in \R,\, \varepsilon \notin \R;\, \varepsilon^2=0\Big\}.
\end{equation*}
The group of units $\A^\times$ is characterized by all dual numbers $z$ with  $N_\A(z)\neq 0$,  where 
\begin{equation*}
N_\A(z)=z \bar{z}= (\Rj(z))^2. \\
\end{equation*}
Thus $z= x+\varepsilon y \in \A^\times$ if and only if $x\neq 0$ and in this case the inverse of $z$ is $$z^{-1}= \frac{1}{x}- \varepsilon \frac{y}{x^2}.$$ Accoring to \cite[Proposition 2.1]{WIC}, Riesz algebra orderings on $\A$ are given by the cones $P^\infty_\alpha$, $( \alpha \in \left[ 0, \infty\right) )$ and $P^{\beta}_{-\infty}$ $( \beta \in \left(-\infty, 0\right]  )$.  The conjugation operator $\sigma(z)=\bar{z}$ is an order and algebra isomorphism between $(\A, P^\infty_0)$ and $( \A, P^0_{-\infty})$, where $P^\infty_0=\{ z= x+ \varepsilon y:\, x, y \geq 0 \}$.\\
 In the sequel, $\A$ is equipped by the Riesz algebra structure with positive cone $P^\infty_0$, and so the absolute value of a dual number $x+\varepsilon y$ is given by $| x+\varepsilon y|= |x|+\varepsilon |y|$. Further, if we let $\|.\|$ denote the lattice euclidean norm in $\R^2$ then (by \cite[Proposition 6.1]{WIC}) a Banach lattice algebra norm in $\A$ is given by
\begin{equation*}
\| z\|_\A= \displaystyle \displaystyle \sup _{\|\mathrm{x}\|=1} \| \mathcal{M}(z).\mathrm{x}\|,\quad \mathcal{M}( x+ \varepsilon y)=\left(\begin{array}{ccc}  x & 0 \\
y& x ×
\end{array}\right).
\end{equation*}
Therefore, an explicit expression of the norm $\|x+\varepsilon y\|_\A$ is given by
\begin{equation}\label{RNA}
\| x+\varepsilon y\|_\A=\displaystyle\sup_{ \theta \in \left[ 0, 2 \pi\right] }  \sqrt{ x^2+ y^2 \sin ^2 \theta+ x y \sin 2\theta}. 
\end{equation}
Just as Gaussian integers, {\it dual Gaussian integers} are defined by
\begin{equation*}
\G_\A=\Z[\varepsilon]=\{ z= x+ \varepsilon y:\, x,y \in \Z\}.
\end{equation*}
It is a ring of quadratic integers in dual numbers, and each  $\upsilon \in \G_\A$ satisfies the identity
\begin{equation} \label{quadricd}
\upsilon^2-2 \Rj(\upsilon)\upsilon+N_\A(\upsilon)=0.
\end{equation}
 
\section{Number systems in hyperbolic  integers}
This section characterize all number systems in hyperbolic integers by considering the complete residue systems
\begin{eqnarray}\label{hdigits}
	\mathfrak{N}_h= \{ \nu \in \ZB:\ 0\leq \nu\leq |q|-1\},	
\end{eqnarray} 
and 
\begin{eqnarray}\label{indigits}
		\mathfrak{N}^{\D}_0= \{0,1, \cdots N_{\D}(|q|)-1\}.
\end{eqnarray} 
The number system $(q, \mathfrak{N}_h)$ appears to be the natural generalization of the usual number systems. Indeed, given  an integral basis $q$ then by uniqueness of the representation, the digits $\nu_i(n) \in \mathfrak{N}_h$ of an ordinary integer $n$ must be natural numbers, since $ \overline{\nu_i(n)}= \nu_i(n) $ for each $i$,  and  selected from the set $\{0, 1, \cdots, |q|-1\}$.
\subsection{Number Systems  $(q, \mathfrak{N}_h)$}
In view of the hyperbolic algorithm division \eqref{algodiv}, the set $\mathfrak{N}_h$ is a complete residue system modulo $q$ if and only if it contains at least two digits, i.e. if and only if $N_{\D}(|q|)\geq 2$. For example, the bases $q$ with $|q|=2$ provide the four idempotent digits $0$, $1$, $\eg_1$ and $\eg_2$.
\begin{theo} \label{th1h} A pair $(q, \mathfrak{N}_h)$ is a  number system for $\ZB$ if and only if $q\leq -2$.	
\end{theo} 
\begin{proof} 
	 Suppose that $(q, \mathfrak{N}_h)$ is a  number system for $\ZB$ and let $\upsilon$ be an arbitrary hyperbolic integer. Then
	\begin{equation}\label{rep1}
	\upsilon= \nu_0+ \cdots+ \nu_k q^k,\quad \nu_i \in \mathfrak{N}_h .
	\end{equation}
	Assuming that $q \notin \D^-$ (the set of negative hyperbolic numbers), then there exists $j\in \{1,2\}$ such that $b=\pi_j(q)>0$.  Let $n \in \Z$ and take $\upsilon=n \eg_j$. By \eqref{rep1} we have	
	\begin{equation*}
	n= a_0+ \cdots+ a_k b^k, \quad a_i =\pi_j( \nu_i) \in \N.
	\end{equation*}
	This yields a contradiction, and so $q$ must be negative with $q\leq -2$ since $\ND(|q|)\geq 2$.\\ 
	Conversely, let $q$ be an hyperbolic integer with $q\leq -2$ and $\upsilon$ be an arbitrary integer, and let $n_i=\pi_i(\upsilon)$, $b_i= \pi_i(q)$ $(i=1,2)$. The $b_i $  are integral bases with $b_i\leq -2$. It follows from usual number systems that we have
	\begin{equation*}
	n_1= r_0+ r_1 b_1+ \cdots+ r_l b_1^l,\quad r_i \in \{ 0,1, \cdots, |b_1|-1\}
	\end{equation*}
	and
	\begin{equation*}
	n_2=  s_0+ s_1 b_2+ \cdots+ s_l b_2^l,\quad s_i \in \{ 0,1, \cdots ,|b_2|-1\},
	\end{equation*}
where some of the leading digits may be zero. Hence,	
	
		\begin{equation*}
	\upsilon= \nu_0+\nu_1 q+ \cdots+ \nu_l q^l, \quad \nu _i= r_i \eg_1+ s_i \eg_2 \in \mathfrak{N}_h.
	\end{equation*}
\end{proof}

\subsection{ Canonical Number Systems $(q, \mathfrak{N}^\D_0)$}
We will formulate the analogue of canonical number systems for complex integers in the setting of hyperbolic integers by considering $\mathfrak{N}^\D_0$ \eqref{indigits} as digits.
\begin{theo}\label{th2h}
	A pair $(q, \mathfrak{N}^\D_0)$ is a canonical number system for $\ZB$ if and only if 
\begin{equation}\label{eqbase1}
q\leq 0,\quad \ND(q)\geq 2 \mbox{~and~} 2 \Ij(q)=\pm 1.
\end{equation}	
\end{theo}
\begin{proof}
	
 Suppose that $\left(q, \mathfrak{N}_0^\D \right) $ is a number system for $\ZB$ and let $\upsilon$ be an arbitrary hyperbolic integer. Then $\upsilon$ has the form
	\begin{equation}\label{reprth1}
	\upsilon=r_0+ \cdots+ r_k q^k, \quad r_i \in \mathfrak{N}_0^\D .
	\end{equation}
A similar reasoning as Theorem \ref{th1h} shows that $q\leq 0$ with   $\ND(q)\geq 2$.  Note that
\begin{equation}\label{identim}
2 \Ij(q^n)= \left( \pi_1(q)\right) ^n-\left( \pi_2 (q)\right) ^n,\quad (n=1,2,\cdots). 
\end{equation}
Then \eqref{identim} together with  \eqref{reprth1} yields that  $\Ij(2q)| \Ij 2 \upsilon$. Take $2\upsilon= 1+\jg $ we get $2 \Ij(q) | 1$, and so $2 \Ij(q)=\pm 1 $.\\
Conversely, let $q=a \eg_1+ b \eg_2 \in \ZB$   satisfying \eqref{eqbase1}. Assume that $2 \Ij(q)=a-b=1$ (the other case follows by considering the conjugate $\bar{q}$, i.e.   by exchanging $a\leftrightarrow b$). Thus 
 \begin{equation} \label{base1}
 q=a\eg_1+ (a-1)\eg_2,\quad a\leq 0,\mbox{~and}~ a^2-a\geq 2.
 \end{equation}
  Let  $\upsilon=n_1 \eg_1+ n_2 \eg_2 \in \ZB$, $n= n_1-n_2 $ and $ m=n_1-(n_1-n_2) a.$ Then
 \begin{equation}\label{Z(q)}
 \upsilon=m+ n q, \quad m,n \in \Z.
 \end{equation}
At this step we sketch the main idea and technique of the proof of \cite[Theorem 1]{Kat-Sab1} to adapt it in the setting of hyperbolic integers. We first  prove that $\upsilon$ can be written as
\begin{equation}\label{repre2}
 \upsilon =u_0+u_1 q+u_2 q^2+ u_3 q^3,\quad u_i \in \N.
\end{equation} 
To do this,  we will use the identity which holds for each  $c \in \Z\setminus\{0\}$ :
\begin{equation}\label{quadraticidentity}
c=|c|\left(  q^2-(2a-1) q+ a^2-a+ \frac{c}{|c|}\right). 
\end{equation}
(Note that (\ref{quadraticidentity}) is a direct consequence of (\ref{quadratich})).\\	
One can clearly assume, without loss of generality, that   $ n, m \in \Z\setminus\{0\}$. Identity \eqref{base1} shows that $-(2a-1)$ and $ a^2-a+ \frac{c}{|c|}$ are positive integers for every $c \in \Z\setminus \{0\}$. Then \eqref{repre2} holds from \eqref{quadraticidentity} by substituting  in \eqref{Z(q)} $n$ and $m$ by
\begin{equation*}
|n|\left(  q^2 -( 2a-1)q+  a^2-a+ \frac{n}{|n|}\right)
\end{equation*}
and
\begin{equation*}
 |m|\left(  q^2 -( 2a-1)q+  a^2-a+ \frac{m}{|m|}\right).
\end{equation*}
It follows from \eqref{repre2}, that there exists an integer $k\geq 3$ such that  
\begin{equation}\label{d-represntation}
\upsilon= d_0+ d_1 q+ \cdots+ d_kq^k,\quad  d_i \in \N .
\end{equation}
 Let $r_0 \in \mathfrak{N}_0^\D$, $t_0 \in \N$ be such that $d_0=r_0+ t_0 N_\D(q)$. Observing that
\begin{equation*}
\ND(q)= q^3-2a q^2+(a^2+a-1)q= a^2-a,
\end{equation*}
one gets from \eqref{d-represntation}
\begin{equation}\label{algo1}
\upsilon= r_0+  \upsilon_1 q , \quad \upsilon_1=d^{(0)}_1+ \cdots+ d^{(0)}_k q^{k-1},
\end{equation}
where 
\begin{eqnarray*}
 d_1^{(0)}&=&d_1+ (a^2+a-1 )t_0,\\ d_2^{(0)}&=&d_2-2at_0,~ d_3^{(0)}=d_3+ t_0,\\ d_j^{(0)}&=&d_j, \quad 4\leq j\leq k.
\end{eqnarray*}
Summing the  $d_j^{(0)}$ over $\{ 1,2, 3,\cdots, k\}$, one has
\begin{equation}\label{aqdk}
d_1^{(0)}+d_2^{(0)}+d_3^{(0)}+ \cdots+ d^{(0)}_k=t_0 \ND(q)+d_1+d_2+d_3+\cdots+ d_k.
\end{equation}
Let $\sigma(\upsilon)= d_0+ \cdots+d_k $ and $ \sigma(\upsilon_1)=d_1^{(0)}+ \cdots+ d^{(0)}_{k}$.
Since $d_0=r_0+ t_0 N_\N(q)$, then from \eqref{aqdk} we have
\begin{equation}\label{algo2}
\sigma(\upsilon)=r_0+\sigma(\upsilon_1).
\end{equation}
Repeating the algorithms \eqref{algo1} and \eqref{algo2},  we get

\begin{equation}\label{eq3th1}
\upsilon_i= r_i+ \upsilon_{i+1} q; \quad\sigma(\upsilon_i)=r_i+\sigma(\upsilon_{i+1}),~ ~ r_i \in \mathfrak{N}_0^\D, 
\end{equation}
 where $\sigma(\upsilon)\geq\sigma(\upsilon_1)\geq \sigma(\upsilon_2)\geq \cdots\geq 0$. This means that there exists an integer $l \in \N$ such that the sequence $(\sigma(\upsilon_i))$ is constant for $i\geq l$, which implies from \eqref{eq3th1} that  $ q^{i-l}$ divides $\upsilon^l$ for all $i\geq l$. This holds only if $\upsilon_l=0$; that is, the algorithm finishes at the step $l$. Hence
\begin{equation*}
\upsilon=r_0+ \cdots+ r_{l-1} q^{l-1}.
\end{equation*}
It remains to prove  that $\mathfrak{N}_0^\D$ is   a complete residue system modulo $q$ ,  i.e. $\mathfrak{N}_0^\D$ has no two elements congruent modulo $q$. Let $s,r \in \mathfrak{N}_0^\D$  be such that   $s\equiv r \mbox{~mod~} q$,i.e. $s-r= \kappa q$ for some $\kappa \in  \ZB$. Then  
\begin{equation}\label{eq5th1}
| s-r|= | \kappa q| \leq  a^2-a-1.
\end{equation}
As $| s-r| \in \N$ means that $ |\kappa| |q|=\pi_1(|\kappa q|) = \pi_2(|\kappa q|)$, we get 
\begin{equation}\label{eq6th1}
|\kappa| |q|=-a \pi_1(|\kappa|)= (1-a) \pi_2(|\kappa|).
\end{equation}
Since $ (-a, 1-a)=1$, then \eqref{eq6th1} yields that  $ \pi_2(|\kappa|)= n (-a)$ for some   $n  \in \N$ and then $\pi_1(|\kappa|)= n(1-a)$. Therefore, from \eqref{eq5th1} we obtain
\begin{equation*}
n( a^2-a) \leq a^2-a -1. 
\end{equation*}
This holds only if $n=0$,  since $a^2-a\geq 2$, and so   $\pi_1(|\kappa|)= \pi_2 (|\kappa|)=0$, which means $\kappa=0$. Hence $s=r$ and this completes the proof. 
	\end{proof}
Theorem \ref{th3h} shows that every hyperbolic Gaussian integer has a unique representation of the form \eqref{reprth1}. However, the bases $q$ are not a hyperbolic Gaussian integers, since $2 \Ij(q)=\pm 1$. The  next theorem characterizes all possible bases $q \in \G_\D$ for which $(q, \mathfrak{N}_0^\D)$ is a canonical number system for $\G_\D$.
 \begin{theo}\label{th3h}
 	Let $q$ be an hyperbolic Gaussian integer. Then, $(q, \mathfrak{N}_0^\D)$ is a canonical number system for hyperbolic Gaussian integers if and only if $q=a\pm \jg$ with $a\leq -2$.	
 \end{theo}
 \begin{proof} 	The proof is entirely analogous  to that of Theorem \ref{th1h}. A base $q$ must be negative with $N_\D(q)\geq 2$ and ( by $\Ij(q)| \Ij(q^n),~ n=1,2, \cdots$) $\Ij(q)=\pm 1$. Thus, $q=a\pm \jg$ with $a\leq 0$ and $a^2\geq 3$; that is, $a\leq -2$. Conversely, we have $\G_\D=\Z[q]$ for each basis $q=a\pm \jg$. Therefore,	using the identities
 \begin{equation*}
 c=|c|\left(  q^2-2a q+a^2-1+ \frac{c}{|c|}\right),\quad c \in \Z\setminus \{0\}  
 \end{equation*}
 and
\begin{equation*}
 \ND(q)= q^3-(2a+1) q^2+(a^2+2a-1)q,
\end{equation*}
 a similar reasoning   shows that each of the bases $q=a\pm \jg$ $(a\leq -2)$ can be used to represent hyperbolic Gaussian integers with $\mathfrak{N}_0^\D=\{ 0, 1, \cdots, a^2-2\}$ as digits set.\\ 
 We prove now  that $\mathfrak{N}_0^\D$ is a complete residue system modulo $q$. Suppose that $q=a+\jg$ (the other case follows by considering the conjugate $\bar{q}$),  and let $s,r \in \mathfrak{N}_0^\D$ be such that $s\equiv r \mbox{~\textrm{mod}~} q$, i.e. $ s-r= \kappa q$ for some $\kappa \in \G_\D $. Then
 \begin{equation}\label{eq2th2}
\ND(\kappa) \ND(q) = (s-r)^2 \leq (a^2-2)^2.
 \end{equation}
As $\pi_1( \kappa q)=\pi_2( \kappa q)$, one obtains
 \begin{equation}\label{eq3th2}
 \Rj( \kappa)=-a \Ij (\kappa).
 \end{equation}
 Thus,  $\ND(\kappa)\ND(q) =(a^2-1)^2 \Ij (\kappa)^2$. Since $a^2>2$, one gets by \eqref{eq2th2}, $(a^2-1) | \Ij(\kappa)| \leq a^2-2$. This holds only if $\Ij(\kappa)=0$, i.e. by \eqref{eq3th2}, $\kappa=0$. Hence $s=r$. 
\end{proof}
\section{Canonical Number systems in dual integers}
In this  section we characterize all canonical number systems $(q, \mathfrak{N}_0^{\A})$ in dual integers, where
\begin{eqnarray*}
\mathfrak{N}^{\A}_0= \{0,1, \cdots , N_{\A}(q)-1\}= \{ 0, 1, \cdots ,\left( \Rj(q)\right)^2-1 \}.
\end{eqnarray*}
 \begin{theo} \label{thd}
 $(q,\mathfrak{N}^{\A}_0 )$ is a number system in $\G_\A$ if and only if $q=a\pm\varepsilon$ with $a\leq -2$.
 \end{theo}
 \begin{proof}
  Let $q=a+\varepsilon b \in \G_A$ be such that $(q,\mathfrak{N}^{\A}_0 )$ is a number system in $\G_A$. Then each  dual Gaussian integer  $\upsilon$  can be expressed as
  	\begin{equation}\label{eq1thd}
  	\upsilon= r_0 + r_1 q+ \cdots+ r_k q^k, \quad r_j \in \mathfrak{N}^{\A}_0.
  	\end{equation} 
  Similarly, we  have  $b= \Ij(q)=\pm 1$. Suppose that $a\geq 0$. It follows from  \eqref {eq1thd} and the identity	$(a+\varepsilon b)^n= a^n+  \varepsilon n a^{n-1} $ $(n=1,2, \cdots)$, that we have
  	\begin{equation*}
  	\Rj(\upsilon)=r_0+r_1 a+ \cdots+ r_k a^n\geq 0.
  	\end{equation*} 
  	This is not possible for every $\upsilon$, and so we must have $a< 0$ which implies  $a\leq -2$, since $N(q)=a^2\geq 2$..\\
  	The converse is proved  using the same method of Theorem \ref{th1h} by considering the corresponding  identities derived from \eqref{quadricd}; that is,  
  	\begin{equation*}
  	c=|c|\left(  q^2-2a q+a^2+ \frac{c}{|c|}\right),\quad c \in \Z\setminus \{0\} ,
  	\end{equation*}
  	and
  	\begin{equation*}
  	\NA(q)= q^3-(2a+1) q^2+(a^2+2a)q.
  	 \end{equation*}
  	 To complete the proof we show that $\mathfrak{N}^\A_0=\{ 0, 1, \cdots, a^2-1\} $ is a  complete residue system modulo $q$ for a basis $q=a+\mu\varepsilon$ with $a\leq -2$, $\mu=\pm 1$. Let  $r, s \in \mathfrak{N}^\A_0 $ be such that $r\equiv s \mbox{~\textrm{mod}~} q $, i.e. $s-r=\kappa q$ for some $\kappa \in \G_\A$. Write $\kappa=k_1+\varepsilon k_2$, then $k_1 a=r-s$ and $k_1 \mu+ k_2 a=0$, which  implies that $ r-s= -\mu k_2 a^2$ and then $|k_2| a^2\leq a^2-1$. This holds only is $k_2=0$. Which gives $r=s$. 
  	 	 \end{proof}
\section{Radix representations  }
In this section we will prove that every hypercomplex number $z \in \mathbb{H}=\A$ or $\D$ is  representable in radix form with respect to number systems $(q, \mathfrak{N})$ in the integers $Z_{\H}$ of $\H$, where $\mathfrak{N}=\mathfrak{N}_h, \mathfrak{N}_0^{\D} $ if $Z_\D= \ZB, \G_\D$ and $\mathfrak{N}=\mathfrak{N}_0^{\A}$ if $\Z_{\A}= \G_\A$. The corresponding   bases are characterized by Theorems \ref{th1h}, \ref{th2h}, \ref{th3h} and \ref{thd}. Our approach is based on the structure of $Z_\H$ as a  closed (for usual topology) full lattice of the plane $\H\equiv \R^2$ and on a topological property of the fundamental domain $\mathcal{F}^{\H}_{(q, \mathfrak{N})} $, where
\begin{equation}\label{fd}
\mathcal{F}^\H_{(q, \mathfrak{N})}:=\{ z \in \H: z=\displaystyle  \sum_{i=1}^{\infty}  \frac{\mu_i}{q^i},~ \mu_i \in  \mathfrak{N}  \}. 
\end{equation}
The corresponding domain in complex number is an arcwise-connected compact set  with a fractal boundary (see \cite{Gilb2}, \cite{Gilb3} and \cite{ATO}). For our purpose we need only to prove compactness of \eqref{fd}.
 \begin{prop} \label{pfd} Let $(q, \mathfrak{N})$ be a number system in $\Z_\H$ (with $q\leq -2$ if $\H=\D$). Then $\mathcal{F}^\H_{(q, \mathfrak{N})} $ is a compact set for usual topology of $\H$.
 		  \end{prop}
 \begin{proof} Let $(q, \mathfrak{N})$ be a number system in $\Z_\H$. 
 Suppose first that $\H=\D$. One then has $\|q^{-1}\|_\D<1$, since $q\leq-2$. This yields that the series  $\sum \mu_i q^{-i}$ is absolute convergent for every sequence $\mu=(\mu_i)_{i\geq 1}$ of elements of $\mathfrak{N}$. Let  
 $(z_n)$ be a sequence of elements of 	$\mathcal{F}_{ ( q,\mathfrak{N})}$. Then  
 	\begin{equation*}
 	z_n= \displaystyle  \sum_{i=1}^{\infty}  \frac{\mu^{(n)}_i}{q^i},~ \mu^{(n)}_i \in  \mathfrak{N}.
 	\end{equation*}
By compactness of $\mathfrak{N} $  we get a subsequence  $(\mu^{(n_k)}_i)$ such that $\mu^{(n_k)}_i$ converges to $\mu_i  \in \mathfrak{N} $ $(i=1,2,\cdots)$. Put 
\begin{equation*}
z= \displaystyle  \sum_{i=1}^{\infty}  \frac{\mu_i}{q^i}
\end{equation*}
Therefore,
\begin{equation*}
\| z_{n_k}-z\|_\D\leq \displaystyle \sup_{i\geq 1}\|\mu^{(n_k)}_i-\mu_i\|_\D \| q^{-1}\|_\D  \|(1-\|q^{-1}\|_\D)^{-1}.
\end{equation*}
This means that there is a subsequence $(z_{n_k})$ of $(z_n)$ that converges to $z$  and this proves the compactness of $\mathcal{F}^\D{(q,\mathfrak{N})} $.\\
In the same manner, one can show  the compactness of $\mathcal{F}^\A_{(q, \mathfrak{N}_0^\A)}$ provided that $\| q^{-1}\|_\A< 1$, where $\|.\|_\A$ is the  Banach lattice algebra norm on $\A$  \eqref{RNA}. Indeed, from Theorem  \ref{thd},  $(q, \mathfrak{N}_0^\A)$ is a number system in $\G_\A$ if and only if $q=a \pm \varepsilon$ with $a\leq -2$,  and then $|q^{-1}|= \frac{1}{|a|}+ \frac{\varepsilon}{a^2}$ with $|a|\geq 2$, that is  $|q^{-1}|\leq \frac{1}{|a|}(1+\varepsilon)$. From (\ref{RNA}), one has $$\|1+\varepsilon\|_\A= \displaystyle \sup_{ \theta \in \left[ 0, 2 \pi\right] }  \sqrt{1+ \cos^2 \theta+ \sin 2\theta }\leq \sqrt{3},$$ and then

\begin{equation*}
\| q^{-1}\|_\A\leq \frac{\sqrt{3}}{|a|} <1,
\end{equation*}
as desired.  
 \end{proof}
 \begin{theo}  Let $(q, \mathfrak{N})$ be a number system in $\Z_\H$ (with $q\leq -2$ if $\H=\D$). Then every $z \in \H$ can be expressed in the form
 	
 	\begin{equation}\label{rh}
 z=	\eta_0+\cdots+ \eta_k q^k+ \displaystyle  \sum_{i=1}^{\infty}  \frac{\mu_i}{q^i} \quad \eta_j, \mu i \in \mathfrak{N}.
 	\end{equation}

 \end{theo}
 \begin{proof}  Let $(q, \mathfrak{N})$ be a number system in $\Z_\H$, and $z \in\H$.  As in \cite[Theorem 2]{Kat-Sab1}   the key idea of the proof is to use the decomposition  $\H=Z_\H+ P_{Z_\H}$, where  $P_{Z_\H}\subset \{ z \in \H:\, 0\leq z\leq 1\}$ is the fundamental parallelepiped of the full lattice $Z_{\H}$ in the plane $\H$. Then for every integer $n\geq 1$ we have 
 	\begin{equation*}
 	q^n z= \upsilon_n+\delta_n;  \quad \upsilon_n \in Z_{\H},~\delta_n \in P_{Z_\H}. 
 	\end{equation*}
Since $(q, \mathfrak{N})$ is a number system in $\Z_\H$,  $\upsilon_n$ can be written as  $$\upsilon_n=\nu_0 +\cdots+ \nu_{d} q^{d},~ d=d(n),~ \quad \nu_i \in \mathfrak{N}.$$
 Let $s(n)=\max\{d(n),n\}$ and  let $ (\nu_i^{(n)})$  be such that $ \nu_i^{(n)}=\nu_{n-i}$ if $i\leq n$ and $\nu_i^{(n)}=0$, otherwise. Therefore, for every integer $n\geq 1$ one has
 \begin{equation} \label{eqz}
z= \vartheta_n+\varsigma_n+ \theta_n , 
 \end{equation}
 where
  	 	\begin{equation*}
 	\vartheta_n= \nu_n+ \cdots \nu_{s(n)} q^{s(n)-n}, ~~~ \varsigma_n= \displaystyle  \sum_{i=1}^{\infty}  \frac{\nu^{(n)}_i}{q^i},~~~\theta_n=\frac{\delta_n}{q^n}.
 	\end{equation*}
 	As $\mathcal{F}^\H_{(q, \mathfrak{N})}$ is a compact set (with $q\leq -2$ if $\H=\D$), we get  a subsequence $ \varsigma_{n_k}$ that converges to $\varsigma \in \mathcal{F}^\H_{(q, \mathfrak{N})}$. Then, from \eqref{eqz} we have for every $k$ 
 	\begin{equation*}
 	z= \vartheta_{n_k}+ \varsigma_{n_k}+\theta_{n_k}.
 	\end{equation*}
 	 It's clear that $\theta_{n_k}\rightarrow 0$, since $\theta_n\rightarrow 0$. And since  $Z_{\H}$ is a closed set (for usual topology), the sequence $(\vartheta_{n_k})$ converges to $\vartheta \in  Z_{\H}$.  Hence
 	\begin{equation*}
 	z=\vartheta+\varsigma,\quad \vartheta \in Z_{\H}, \; \; \varsigma \in \mathcal{F}^\H_{(q, \mathfrak{N})}.
 	\end{equation*}
 	This completes the proof.
 	  \end{proof} 
 
 \medskip

 	  {\large \bf Data Availability Statement.} No data were used to support this study.

\end{document}